\documentclass[12pt]{article}
\usepackage{amsmath}
\usepackage{amssymb}
\usepackage{amsmath,amssymb,amsbsy,amsfonts,amsthm,latexsym,
amsopn,amstext,amsxtra,euscript,amscd}

\begin{document}

\def\A{\mathbb{A}}
\def\B{\mathbf{B}}
\def \C{\mathbb{C}}
\def \F{\mathbb{F}}
\def \K{\mathbb{K}}

\def \Z{\mathbb{Z}}
\def \P{\mathbb{P}}
\def \R{\mathbb{R}}
\def \Q{\mathbb{Q}}
\def \N{\mathbb{N}}
\def \Z{\mathbb{Z}}

\def\B{\mathcal B}
\def\e{\varepsilon}

\def\cA{{\mathcal A}}
\def\cB{{\mathcal B}}
\def\cC{{\mathcal C}}
\def\cD{{\mathcal D}}
\def\cE{{\mathcal E}}
\def\cF{{\mathcal F}}
\def\cG{{\mathcal G}}
\def\cH{{\mathcal H}}
\def\cI{{\mathcal I}}
\def\cJ{{\mathcal J}}
\def\cK{{\mathcal K}}
\def\cL{{\mathcal L}}
\def\cM{{\mathcal M}}
\def\cN{{\mathcal N}}
\def\cO{{\mathcal O}}
\def\cP{{\mathcal P}}
\def\cQ{{\mathcal Q}}
\def\cR{{\mathcal R}}
\def\cS{{\mathcal S}}
\def\cT{{\mathcal T}}
\def\cU{{\mathcal U}}
\def\cV{{\mathcal V}}
\def\cW{{\mathcal W}}
\def\cX{{\mathcal X}}
\def\cY{{\mathcal Y}}
\def\cZ{{\mathcal Z}}

\def\f{\frac{|\A||B|}{|G|}}
\def\AB{|\A\cap B|}
\def \Fq{\F_q}
\def \Fqn{\F_{q^n}}

\def\({\left(}
\def\){\right)}
\def\fl#1{\left\lfloor#1\right\rfloor}
\def\rf#1{\left\lceil#1\right\rceil}
\def\Res{{\mathrm{Res}}}

\newcommand{\comm}[1]{\marginpar{
\vskip-\baselineskip \raggedright\footnotesize
\itshape\hrule\smallskip#1\par\smallskip\hrule}}

\newtheorem{lem}{Lemma}
\newtheorem{lemma}[lem]{Lemma}
\newtheorem{prop}{Proposition}
\newtheorem{proposition}[prop]{Proposition }
\newtheorem{thm}{Theorem}
\newtheorem{theorem}[thm]{Theorem}
\newtheorem{cor}{Corollary}
\newtheorem{corollary}[cor]{Corollary}
\newtheorem{prob}{Problem}
\newtheorem{problem}[prob]{Problem}
\newtheorem{ques}{Question}
\newtheorem{question}[ques]{Question}
\newtheorem{rem}{Remark}

\title{Double exponential sums and congruences with intervals and exponential functions modulo a prime}

\author{
{\sc M.~Z.~Garaev} }

\date{}

\maketitle

\begin{abstract}
Let $p$ be a large prime number and $g$ be any integer of multiplicative order $T$ modulo $p$. We obtain a new estimate of the double
exponential sum
$$
S=\sum_{n\in \mathcal{N}}\left|\sum_{m\in \mathcal{M} }e_p(an g^{m})\right|, \quad \gcd (a,p)=1,
$$
where $\mathcal{N}$ and $\mathcal{M}$ are intervals of consecutive integers with $|\mathcal{N}|=N$ and $|\mathcal{M}|=M<T$ elements. 
One representative
example is the following consequence of the main result:  if $N=M\approx p^{1/3}$, then  $|S|< N^{2-1/8 + o(1)}$.
We then apply our  estimate  to obtain new results on  additive congruences involving intervals and exponential functions.
\end{abstract}

\paragraph{Mathematical Subject Classification:}  11L07, 11L79

\paragraph{Keywords:} exponential sums, exponential functions, congruences

\section{Introduction}

Let $p$ be a large prime number, $g$ be an integer with $\gcd(g,p)=1$. Denote by $T$ the multiplicative order
of $g$ modulo $p.$ Let
$$
\cN = \{u+1,\ldots, u+N\}\quad {\rm and} \quad \cM = \{v+1,\ldots, v+M\}
$$
be two intervals of consecutive integers with
$$
|\cN|=N\le p\quad  {\rm and}\quad  |\cM|=M\le T.
$$
In the present paper we are concerned with the problem of upper bound estimates for the double exponential sum
$$
S_{a,p,g}(\bar{\alpha},\bar{\beta}; \cN,\cM)=\sum_{n\in \cN}\sum_{m\in \cM}\alpha_n\beta_m e_p(ang^m),\quad \gcd(a,p)=1,
$$
where $\alpha_n$ and $\beta_m$ are complex coefficients with $|\alpha_n|, \beta_m|\le 1$, and  $e_p(z)=e^{2\pi i z/p}$. Here,
for a negative integer $-k$, the number $g^{-k}$ is defined to be an integer with $g^{-k}g^k\equiv 1 \!\!\! \pmod  p.$

In the special case $\alpha_n=\beta_m=1$, the sum $S_{a,p}(\bar{\alpha},\bar{\beta}; \cN,\cM)$ has appeared in the work of Bourgain~\cite{Bourg1},
where he has estimated it for very short intervals $\cN$ and $\cM.$

When  $M>p^{1/2}$, one can apply classical estimates of single sums with exponential functions
which would lead to nontrivial
bounds for $S_{a,p}(\bar{\alpha},\bar{\beta}; \cN,\cM)$ with reasonably good power savings. On the other hand,
from the celebrated work of Bourgain, Glibichuk and Konyagin~\cite{BGK}, it follows that for any $\varepsilon>0$ there exists $\delta>0$
such that if $N>p^{\varepsilon}, M>p^{\varepsilon}$, then
$$
\Bigl|S_{a,p,g}(\bar{\alpha},\bar{\beta}; \cN,\cM)\Bigr|< NMp^{-\delta}.
$$
However, the value  of $\delta$ in terms of $\varepsilon$ is very small, and, in particular, it does not give good savings
for medium sized intervals.

Recently,  Shparlinski and Yau~\cite{Shp1} obtained a number of new explicit estimates on $S_{a,p}(\bar{\alpha},\bar{\beta}; \cN,\cM)$.
One of the features of~\cite{Shp1} is that  some of the estimates given there work well for certain ranges of $N$ and $M$ below the critical value $p^{1/2}$.
They also noted, that the work of Roche-Newton,  Rudnev and Shkredov~\cite{RRShk} leads to certain nontrivial bounds in the
range $M>p^{1/3+c}$, for any positive constant $c$. Nevertheless,  in some very interesting cases  (for example, if $N, M<p^{1/3}$) these estimates  become trivial
and one naturally asks what can be done in these cases. 

Note that by the Cauchy-Schwarz inequality,
\begin{equation*}
\begin{split}
|S_{a,p,g}(\bar{\alpha},\bar{\beta}; \cN,\cM)|^2\le M \sum_{m\in\cM}\left|\sum_{n\in \cN}\alpha_n e_p(ang^m)\right|^2 \\ \le M \sum_{n_1\in \cN}\sum_{n_2\in \cN}\left|\sum_{m\in  \cM}e_p(a(n_1-n_2)g^m)\right|\\ \ll MN \sum_{0\le n\le N}\left|\sum_{m\in  \cM}e_p(ang^m)\right|.
\end{split}
\end{equation*}
Thus, in what follows we shall concentrate our attention on the sum 
$$
S_{a,p,g}(\cN,\cM)=\sum_{n\in \cN}\left|\sum_{m\in  \cM}e_p(ang^m)\right|.
$$
In the present paper we obtain  a new explicit estimate for $S_{a,p}(\cN,\cM)$ 
which, in particular, is nontrivial in the range $N=M>p^{2/7+c},$ for any constant $c>0$. Then we apply our bound to obtain new results on congruences
involving intervals and exponential functions.

\bigskip

{\bf Notations.}   In what follows, we use the notation $A\lesssim B$ to mean that
$|A|<B p^{o(1)}$, or equivalently,  for any $\varepsilon>0$ there
is a constant $c=c(\varepsilon)$ such that
$|A|<c Bp^{\varepsilon}.$  Given two sets $\cX$ and $\cY$ their
product-set $\cX\cdot\cY$ and the sum-set $\cX+\cY$ are defined by
$$
\cX\cdot\cY=\{ab;\quad a\in \cX,\, b\in \cY\},\quad \cX+\cY=\{a+b;\quad a\in \cX,\, b\in \cY\}.
$$
As usual, for a positive integer $k$, the $k$-fold sum-set $k\cX$ is defined by
$$
k\cX = \{a_1+\ldots+a_k;\quad a_k\in \cX\}.
$$ 
The notation $|\cX|$ stands for the cardinality of the set $\cX$.

\section{Our results}

\begin{theorem} 
\label{thm:Main} Let $M<p^{2/3}$. Then
$$
S_{a,p,g}(\cN,\cM)=\sum_{n\in \cN}\left|\sum_{m\in  \cM}e_p(ang^m)\right|= NM\Delta,
$$
where
$$
\Delta \lesssim \frac{1}{M^{3/8}} +
 \Bigl(\frac{p}{NM^{5/2}}\Bigr)^{1/4} + \Bigl(\frac{p}{N^{4/3}M^{7/3}}\Bigr)^{3/16}+\Bigl(\frac{p}{N^{2}M^{3/2}}\Bigr)^{1/4}.
$$
\end{theorem}

In particular,  if $N=M\approx p^{1/3}$, then $S_{a,p,g}(\cN,\cM)\lesssim N^{2-1/8}.$

It is well-known that nontrivial exponential sum estimates is a basic tool in investigation of additive problems.  Theorem~\ref{thm:Main}
has the following application.

\begin{theorem} 
\label{thm:10 sums} Let $\varepsilon>0$ be a fixed small positive constant and let $\cN_i$ and $\cM_i$  be intervals of consecutive integers with 
$|\cN_i|=N_i$ and $|\cM_i|=M_i$, satisfying
$$
p \ge N_i>p^{1/3+\varepsilon}, \quad T\ge M_i>p^{1/3+\varepsilon},\quad , i=1,2,\ldots, 10.
$$
Then for any integer $\lambda$ the number $J_{10}$ of solutions of the congruence
$$
x_1g^{y_1}+\ldots +x_{10}g^{y_{10}}\equiv \lambda \!\!\! \pmod  p, \quad x_i\in\cN_i, \, y_i\in\cM_i, 
$$
satisfies 
$$
J_{10}=\frac{\prod_{i=1}^{10}(N_iM_i)}{p}\Bigl(1+O(p^{-\delta})\bigr),\qquad \delta=\delta(\varepsilon)>0.
$$
\end{theorem}

\bigskip

Theorem~\ref{thm:10 sums} provides with the asymptotic formula for the number of solutions of the congruence. But if one is interested
only on the question of solubility, then our intermediate result that we obtain in the course of the proof of Theorem~\ref{thm:Main}
combined with the result of Glibichuk~\cite{Glib} and Roche-Newton, Rudnev and Shkredov~\cite{RRShk}, leads to the following results.

\begin{theorem} 
\label{thm:8 sums} Let $\varepsilon>0$ be a fixed small positive constant and let 
$$
N>p^{1/3+\varepsilon},\quad M>p^{1/3+\varepsilon}.
$$ 
Then any integer $\lambda$ modulo $p$ can be represented  in the form
$$
x_1g^{y_1}+\ldots +x_{8}g^{y_{8}}\equiv \lambda \!\!\! \pmod  p,
$$
for some $x_i\in\cN$ and $y_i\in\cM$.
\end{theorem}

\begin{theorem} 
\label{thm:16 sums} Let $\varepsilon>0$ be a fixed small positive constant and let 
$$
N>p^{2/7+\varepsilon},\quad M>p^{2/7+\varepsilon}.
$$ 
Then any integer $\lambda$ modulo $p$ can be written in the form
$$
x_1g^{y_1}+\ldots +x_{16}g^{y_{16}}\equiv \lambda \!\!\! \pmod  p,
$$
for some $x_i\in\cN$ and $y_i\in\cM$.
\end{theorem}

\section{Lemmas}

The following lemma is contained in~\cite{CilGar} under the additional restriction $|\cU|<p^{2/5}$. This restriction has been removed
in~\cite{Gar1}. Note that when $\cU$ is a subgroup of size $|\cU|>p^{1/2}$,
the statement follows from~\cite[Theorem 1]{BKSh}.
\begin{lemma}
\label{lem:CG+G}
Let $H$ be a positive integer  and let $\cU\subset
\F_p^*$ be such that
$$
|\cU\cdot\cU|< 10|\cU|.
$$
Then the number $J_0$ of solutions of the congruence
$$
xr\equiv x_1 r_1\!\!\! \pmod  p;\quad x,x_1\in \mathbb{N},\quad x,x_1\le H,\quad r,r_1\in\cU
$$
satisfies
$$
J_0\lesssim |\cU|H+\frac{|\cU|^2H^2}{p}+\frac{|\cU|^{7/4}H}{p^{1/4}}.
$$
\end{lemma}

\begin{lemma}
\label{lem1:CG+G}
The number $J$ of solutions of the congruence
\begin{equation}
\label{eqn:lem1CG+G}
xg^y\equiv x_1 g^{y_1}\!\!\! \pmod  p; \quad x,x_1\in \cN,\quad y, y_1\in\cM
\end{equation}
satisfies
$$
J\lesssim M^2+MN +\frac{M^2N^2}{p}+\frac{M^{7/4}N}{p^{1/4}}.
$$
\end{lemma}

\begin{proof}
Given $y,y_1\in\cM$, denote by $J(y,y_1)$ the number of solutions of ~\eqref{eqn:lem1CG+G} in variables $x,x_1\in\cN$. Then
$$
J=\sum_{y,y_1}J(y,y_1).
$$
We can restrict the summation to those $y,y_1\in\cM$ for which $J(y,y_1)\not=0$. Thus, there is a pair $x_0, x_{0}'\in \cN$ depending on $y,y_1$
such that 
$$
(x-x_0)g^y\equiv (x_1-x_{0}') g^{y_1}\!\!\! \pmod  p.
$$ 
Note that $|x-x_0|\le N,\, |x_1-x_{0}'|\le N$. Hence,
$$
J(y,y_1)\le J_0(y,y_1),
$$
where $J_0(y,y_1)$ is the number of solutions of the congruence
$$
xg^y\equiv x_1 g^{y_1}\!\!\! \pmod  p,\quad |x|, |x_1|\le N. 
$$
Thus,
$$
J\le \sum_{y,y_1}J_0(y,y_1)=J_0',
$$
where $J_0'$ is the number of solutions of the congruence
$$
xg^y\equiv x_1 g^{y_1}\!\!\! \pmod  p,\quad |x|, |x_1|\le N, \quad y,y_1\in \cM. 
$$
If $x=0$, then $x_1=0$ and this case contributes to $J_0'$ the quantity $M^2$. If $x\not=0$, then $x_1\not=0$. In this case we 
apply the Cauchy-Schwarz inequality
to symmetrize the equation, and then the desired result follows from application of Lemma~\ref{lem:CG+G} to the set 
$U=\{g^{y}\!\!\! \pmod  p, \,y\in\cM\}.$

\end{proof}

\begin{corollary}
\label{cor1:CG+G}
The following bound holds:
$$
\Bigl|\{xg^y\!\!\! \pmod  p,  \, x\in \cN,\, y\in\cM\}\Bigr|\, \gtrsim \,\min\{N^2, NM, p, p^{1/4}M^{1/4}N\}.
$$
\end{corollary}

We need the following result of Roche-Newton,  Rudnev and Shkredov~\cite{RRShk}.

\begin{lemma}
\label{lem:RRSh}
Let $M<p^{2/3}$. Then the number $E_{+}(g^{\cM})$ of solutions of the congruence
$$
g^{m_1}+g^{m_2}\equiv g^{m_3}+g^{m_4}\!\!\! \pmod  p,\quad m_1,m_2,m_3,m_4\in \cM,
$$
satisfies $E_{+}(g^{\cM})\ll M^{5/2}$.
\end{lemma}

We also need  the following result of Glibichuk~\cite{Glib}.
\begin{lemma}
\label{lem:Glib}
Let $\cX, \cY \subset \F_p $ be such that  $|\cX||\cY|>2p$.  Then $8(\cX\cdot \cY)=\F_p$.
\end{lemma}

\section{Proof of Theorem~\ref{thm:Main}}

Let $\gamma_n$ be complex numbers such that $|\gamma_n|=1$ and
$$
\left|\sum_{m\in  \cM}e_p(ang^m)\right|=\gamma_n \sum_{m\in  \cM}e_p(ang^m)
$$
 We have that
 $$
S_{a,p,g}(\cN,\cM)=\sum_{n\in \cN}\left|\sum_{m\in  \cM}e_p(ang^m)\right|= \sum_{n\in \cN}\sum_{m\in  \cM}\gamma_n e_p(ang^m).
 $$
We recall that $\cM=\{v+1,\ldots, v+M\}$. As in the work of Friedlander and Iwaniec~\cite{FrIw}, we introduce the function
$$
f(x)=\left\{
\begin{array}{lll}
0 &if & x\le v;\\
x-v &if & v \le x\le v+1;\\
1 & if& v+1\le x\le v+M;  \\
v+M+1-x &if & v+M\le x\le v+M+1;\\
0 &if & x\ge  v+M+1.
\end{array}\right.
$$
Let $F$ denote the Fourier transform of $f$,
$$
F(y)=\int_{-\infty}^{\infty}f(x)e^{-2\pi i yx}dx.
$$
Then integrating by part it follows that $|F(y)|\le \min \{M, |\pi y|^{-1}, |\pi y|^{-2}\}$. Hence,
\begin{equation}
\label{eqn:Fint}
\int_{-\infty}^{\infty}|F(y)|dy\le  \int_{0}^{1/M}Mdy+\int_{1/M}^{1}|\pi y|^{-1}dy+\int_{1}^{\infty}|\pi y|^{-2}dy\ll \log M.
\end{equation}
Therefore, since $f(x)=\int_{-\infty}^{\infty}F(y)e^{2\pi i yx}dy$, we get that
\begin{equation*}
\begin{split}
S_{a,p,g}(\cN,\cM)=&\sum_{n\in \cN}\sum_{m\in\cM}\gamma_n e_p(a n g^m)=\\&\sum_{n\in\cN}\sum_{m=-\infty}^{\infty}\gamma_n f(m)e_p(a n g^m)=\\
&\frac{1}{M}\sum_{k=1}^{M}\sum_{n\in\cN}\sum_{m=v-M}^{v+M-1}\gamma_n f(k+m)e_p(a n g^kg^m)=\\
&\frac{1}{M}\int_{-\infty}^{\infty}F(y)\Bigl(\sum_{k=1}^{M}\sum_{n\in\cN}\sum_{m=v-M}^{v+M-1}\gamma_n e^{2\pi i (k+m)y}e_p(a n g^kg^m) \Bigr)dy.
\end{split}
\end{equation*}
Hence, in view of~\eqref{eqn:Fint}, for some $y\in \R$ we have that
$$
S_{a,p,g}(\cN,\cM)\lesssim \frac{1}{M}\Bigl|\sum_{k=1}^{M}\sum_{n\in\cN}\sum_{m=v-M}^{v+M-1}\gamma_n \delta_k\delta_m
e_p(a n g^kg^m) \Bigr|,
$$
where $\delta_j=e^{2\pi i j y}$. Thus,
\begin{equation}
\label{eqn:SW2}
S_{a,p,g}(\cN,\cM)\lesssim \frac{W}{M},
\end{equation}
where 
$$
W=\sum_{k=1}^{M}\sum_{n\in \cN}\Bigl|\sum_{m=v-M}^{v+M-1}\delta_m
e_p(a n g^kg^m) \Bigr|.
$$
Applying the Cauchy-Schwarz inequality, we obtain that
\begin{equation*}
\begin{split}
W^2\le NM& \sum_{k=1}^{M}\sum_{n\in\cN}\Bigl|\sum_{m=v-M}^{v+M-1}\delta_m
e_p(a n g^kg^m) \Bigr|^2=\\ 
NM&\sum_{k=1}^{M}\sum_{n\in\cN}\sum_{m_1=v-M}^{v+M-1}\sum_{m_2=v-M}^{v+M-1}\delta_{m_1}\bar \delta_{m_2}
e_p(a n g^k(g^{m_1}-g^{m_2})) \le \\ NM&\sum_{m_1=v-M}^{v+M-1}\sum_{m_2=v-M}^{v+M-1}\Bigl|\sum_{k=1}^{M}\sum_{n\in\cN}
e_p(a n g^k(g^{m_1}-g^{m_2})) \Bigr|.
\end{split}
\end{equation*}
Thus, if we denote by $I_{\lambda}$ the number of solutions of the congruence
$$
g^{m_1}-g^{m_2}\equiv \lambda\!\!\! \pmod  p, \quad v-M\le m_1,m_2\le v+M-1,
$$
we get that
$$
W^2\le NM \sum_{\lambda=0}^{p-1}I_{\lambda}\Bigl|\sum_{k=1}^{M}\sum_{n\in\cN}
e_p(a \lambda n g^k) \Bigr|.
$$
Applying Cauchy-Schwarz inequality again, we get that
\begin{equation}
\label{eqn:W24}
W^4\le N^2M^2\Bigl( \sum_{\lambda=0}^{p-1}I_{\lambda}^2\Bigr) \sum_{\lambda=0}^{p-1}\Bigl|\sum_{k=1}^{M}\sum_{n\in\cN}
e_p(a \lambda n g^k) \Bigr|^2
\end{equation}
The quantity $\sum_{\lambda=0}^{p-1}I_{\lambda}^2$ is equal to the number of solutions of the congruence
$$
g^{m_1}-g^{m_2}\equiv g^{m_3}-g^{m_4}\!\!\! \pmod  p, \quad v-M\le m_1,m_2, m_3,m_4\le v+M-1.
$$
It then follows from Lemma~\ref{lem:RRSh} that
\begin{equation}
\label{eqn:M5/2}
\sum_{\lambda=0}^{p-1}I_{\lambda}^2\ll M^{5/2}.
\end{equation}
Furthermore,
$$
\sum_{\lambda=0}^{p-1}\Bigl|\sum_{k=1}^{M}\sum_{n\in\cN}
e_p(a \lambda n g^k) \Bigr|^2=pJ,
$$
where $J$ is the number of solutions of the congruence
$$
ng^k\equiv n_1g^{k_1}\!\!\! \pmod  p, \quad 1\le n,n_1\le N,\quad 1\le k,k_1\le M.
$$
Applying Lemma~\ref{lem1:CG+G}, we get that
$$
\sum_{\lambda=0}^{p-1}\Bigl|\sum_{k=1}^{M}\sum_{n\in\cN}
e_p(a \lambda n g^k) \Bigr|^2\lesssim p\Bigl(M^2+MN+\frac{M^2N^2}{p}+\frac{M^{7/4}N}{p^{1/4}}\Bigr).
$$
Inserting this  and~\eqref{eqn:M5/2} into~\eqref{eqn:W24}, we obtain that
$$
\frac{W^4}{M^4}\lesssim pN^2M^{5/2} + pN^3M^{3/2}+N^4M^{5/2}+ N^{3}M^{9/4}p^{3/4}.
$$
Thus,
$$
\frac{W}{M}\lesssim NM\Bigl(\frac{1}{M^{3/8}} +
 \Bigl(\frac{p}{NM^{5/2}}\Bigr)^{1/4} + \Bigl(\frac{p}{N^{4/3}M^{7/3}}\Bigr)^{3/16}+\Bigl(\frac{p}{N^{2}M^{3/2}}\Bigr)^{1/4}\Bigr).
$$
Substituting this in~\eqref{eqn:SW2}, we conclude the proof.

\section{Proofs of Theorems~\ref{thm:10 sums},~\ref{thm:8 sums} and~\ref{thm:16 sums}}

We start with the proof of Theorem~\ref{thm:10 sums}. First of all we note that if $M>p^{2/3}$, then from the classical bounds 
of exponential sums with exponential functions, we know that
$$
\max_{\gcd(a,p)=1}\Bigl|\sum_{m\in \cM}e_p(ag^m)\Bigr|\lesssim p^{1/2}.
$$
Hence, in this case we have that 
$$
\Bigl|\sum_{n\in \cN}\sum_{m\in \cM}e_p(ang^m)\Bigr|\lesssim M+Np^{1/2+o(1)}\lesssim NMp^{-1/6}.
$$
If $M<p^{2/3}$, then by Theorem~\ref{thm:Main}, we get that
\begin{equation}
\label{eqn:estanycase}
\Bigl|\sum_{n\in \cN}\sum_{m\in \cM}e_p(ang^m)\Bigr|<\frac{NM}{p^{1/24+\delta_0}},\quad \delta_0=\delta_0(\varepsilon)>0.
\end{equation}
Thus, the estimate~\eqref{eqn:estanycase} holds. Expressing $J_{10}$ in terms of exponential sums, we get
$$
J_{10}=\frac{1}{p}\sum_{a=0}^{p-1}\prod_{j=1}^{10}\Bigl(\sum_{x\in\cN_j}\sum_{y\in\cM_j}e_p(axg^y)\Bigr)e_p(-a\lambda).
$$
Separating the term that corresponds to $a=0$ and then using~\eqref{eqn:estanycase}, we obtain that
\begin{equation}
\label{eqn: J minus main term}
\Bigl|J-\frac{\prod_{j=1}^{10}(N_jM_j)}{p} \Bigr|< \frac{\prod_{j=1}^{8}(N_jM_j)}{p^{1/3+8\delta_0}}R
\end{equation}
where
$$
R=\frac{1}{p}\sum_{a=1}^{p-1}\Bigl|\sum_{x\in\cN_9}\sum_{y\in\cM_9}e_p(axg^y)\Bigr|\Bigl|\sum_{x\in\cN_{10}}\sum_{y\in\cM_{10}}e_p(axg^y)\Bigr|. 
$$
Applying the Cauchy-Schwarz inequality, we get that
$$
R\le \sqrt{\frac{1}{p}\sum_{a=0}^{p-1}\Bigl|\sum_{x\in\cN_9}\sum_{y\in\cM_9}e_p(axg^y)\Bigr|^2}\times \sqrt{\frac{1}{p}\sum_{a=0}^{p-1}\Bigl|\sum_{x\in\cN_{10}}\sum_{y\in\cM_{10}}e_p(axg^y)\Bigr|^2}.
$$
Furthermore,
$$
\frac{1}{p}\sum_{a=0}^{p-1}\Bigl|\sum_{x\in\cN_j}\sum_{y\in\cM_j}e_p(axg^y)\Bigr|^2 = T_j,
$$
where $T_j$ is the number of solutions of the congruence
$$
xg^y\equiv x_1g^{y_1}\!\!\! \pmod  p,\quad x,x_1\in \cN_j,\quad y,y_1\in\cM_j.
$$
Thus,
$$
R\le \sqrt{T_9 T_{10}}.
$$
From Lemma~\ref{lem1:CG+G} it follows that
$$
T_j\lesssim N_j^2M_j^2\Bigl(\frac{1}{N_j^2}+\frac{1}{N_jM_j}+\frac{1}{p}+\frac{1}{p^{1/4}M_j^{1/4}N_j}\Bigr)<\frac{N_j^2M_j^2}{p^{2/3+1.1\varepsilon}}.
$$
Hence,
$$
R\le \sqrt{T_9 T_{10}}\le \frac{N_9M_9N_{10}M_{10}}{p^{2/3+\varepsilon}}.
$$
Inserting this into~\eqref{eqn: J minus main term}, we get that
$$
\Bigl|J-\frac{\prod_{j=1}^{10}(N_jM_j)}{p} \Bigr|< \frac{\prod_{j=1}^{10}(N_jM_j)}{p}p^{-\delta},\quad \delta=\delta(\varepsilon)>0.
$$
This finishes the proof of Theorem~\ref{thm:10 sums}.

In order to prove Theorem~\ref{thm:8 sums}, let $M_1=\lfloor0.5M\rfloor$ and define the sets $\cX$ and $Y$ as follows:
$$
\cX=\{xg^y\!\!\! \pmod  p,\, x\in\cN, y\in \cM_1\},\quad \cY=\{g^y\!\!\! \pmod  p,\,  1\le y\le M_1\},
$$
where $\cM_1=\{v+1, v+2,\ldots, v+M_1 \}$. By Corollary~\ref{cor1:CG+G}, we have 
$$
|\cX||\cY|\ge p^{2/3}p^{1/3+0.5\varepsilon}>2p.
$$
Therefore, by Lemma~\ref{lem:Glib}, $8(\cX\cdot \cY)=\F_p$. Since
$$
\cX\cY\subset  \{xg^y\!\!\! \pmod  p,\, x\in\cN,  y\in \cM\}, 
$$
the claim follows.

The proof of Theorem~\ref{thm:16 sums} is similar. We can assume that $M<p^{2/3}$. Define $M_1, \cM_1$ and  $\cX$  as in the proof of Theorem~\ref{thm:8 sums},
and define $\cY$ by
$$
\cY=\{g^{y_1}+g^{y_2}\!\!\! \pmod  p,\,  1\le y_1,y_2\le M_1\}.
$$
From Corollary~\ref{lem1:CG+G} we have that $|\cX|\ge p^{4/7+\varepsilon}$. 
Also from Lemma~\ref{lem:RRSh} and the relationship between the number of solutions and the cardinality,
it follows that $|\cY|\gg M^{3/2}$. Hence,
$$
|\cX||\cY|\ge p^{4/7+\varepsilon} p^{3/7}>2p.
$$
Hence, by Lemma~\ref{lem:Glib}, $8(\cX\cdot \cY)=\F_p$. Since
$$
\cX\cY\subset  \{x_1g^{y_1}+x_2g^{y_2}\!\!\! \pmod  p,\, x_1,x_2\in\cN,  y_1,y_2\in \cM\}, 
$$
the result follows.

Address of the author:\\

M.~Z.~Garaev, Centro de Ciencias Matem\'{a}ticas,  Universidad
Nacional Aut\'onoma de M\'{e}xico, C.P. 58089, Morelia,
Michoac\'{a}n, M\'{e}xico,

Email: {\tt garaev@matmor.unam.mx}

\end{document}